\documentclass[12pt,twoside]{amsart}

\usepackage{amsmath,amsthm,amscd,amssymb,mathrsfs,graphicx,amsfonts,mathrsfs}
\usepackage{amsfonts}
\usepackage{amssymb,enumerate}
\usepackage{amsthm}
\usepackage[all]{xy}
\usepackage{indentfirst}
\usepackage{hyperref}
\allowdisplaybreaks[4] \footskip=12pt
\renewcommand{\uppercasenonmath}[1]{}

\numberwithin{equation}{section} \theoremstyle{plain}
\newtheorem*{thm*}{Main Theorem}
\newtheorem{thm}{Theorem}[section]
\newtheorem{cor}[thm]{Corollary}
\newtheorem*{cor*}{Corollary}
\newtheorem{lem}[thm]{Lemma}
\newtheorem*{lem*}{Lemma}

\newtheorem*{fact*}{Fact}

\newtheorem*{nota*}{Notation}
\newtheorem{prop}[thm]{Proposition}
\newtheorem*{prop*}{Proposition}

\newtheorem*{rem*}{Remark}

\newtheorem*{observation*}{Observation}

\newtheorem*{exa*}{Example}

\newtheorem*{df*}{Definition}

\newtheorem*{con*}{Construction}

\renewcommand{\geq}{\geqslant}
\renewcommand{\leq}{\leqslant}

\begin{document}
\begin{center}
{\large  \bf Improved Bounds for the  $s$-multiplicity}

\vspace{0.5cm} Zhongkui Liu$^{[1]}$,
    Junquan Qin$^{[1]}$\footnote{Corresponding author. E-mails:  Liuzk@nwnu.edu.cn, qinjunquan2018@163.com, yangxy@zust.edu.cn}, and Xiaoyan Yang$^{[2]}$\\

%\bigskip1
College of Mathematics and Statistics, Northwest Normal University, Lanzhou 730070, P. R. China $^{[1]}$

 School of Science, Zhejiang University of Science and Technology
Hangzhou 310023,  P. R. China$^{[2]}$
\end{center}
\bigskip
%\bigskip

$Abstract$. Let $(R, \mathfrak{m}_R )$, $(S, \mathfrak{m}_S)$ and $(T, \mathfrak{m}_T)$ be Noetherian local rings  sharing the same residue field $k$ and prime characteristic $p>0$. We establish some formulas relating the $h$-function and $s$-multiplicity of the fiber product $R \times_T S$  in terms of the $h$-functions and $s$-multiplicities of $R$, $T$ and $ S$.  Furthermore, we derive formulas that connect  the $h$-function and $s$-multiplicity of the idealization ring $R \ltimes M$  to the corresponding invariants  of $R$ and $ M$, where $M$ is a finitely generated $R$-module. As  applications of these results, we derive new estimates  for the Taylor-Miller question and the Watanabe-Yoshida conjecture concerning $s$-multiplicity.  \\
\vbox to 0.2cm{}\\\leftskip0truemm \rightskip0truemm
{\it KeyWords:} $h$-function; $s$-multiplicity; fiber product; idealization ring\\
\leftskip0truemm \rightskip0truemm
\leftskip0truemm \rightskip0truemm
{\it2020 Mathematics Subject Classification:}  13D40; 13H15; 13H05
\leftskip0truemm \rightskip0truemm
\bigskip
\section{\bf Introduction}
%\bigskip
 Hilbert-Kunz multiplicity and Hilbert-Samuel multiplicity are numerical invariants that play a central role in commutative algebra and algebraic geometry over fields of prime characteristic. Owing to their numerous remarkable properties, these invariants have been extensively developed and applied since their introduction.

Taylor \cite{T} proposed two new numerical invariants in a local ring $(R, \mathfrak{m})$ of  positive characteristic $p$: the $h$-function $h_{s}(I,J;M)$ and $s$-multiplicity $e_{s}(I,J;M)$ for any $\mathfrak{m}$-primary ideals $I$ and $J$. The  $s$-multiplicity is obtained by normalizing  $h_{s}(I,J;M)$ via the  $h$-function $h_{s}(R)$ of   regular local rings. Furthermore,  for large real values of $s$, the $s$-multiplicity coincides with the Hilbert-Kunz multiplicity of $J$,  while for small real values of $s$,   it agrees with  the Hilbert-Samuel multiplicity of $I$. Thus, the s-multiplicity serves as an  interpolate between the Hilbert-Kunz multiplicity and Hilbert-Samuel multiplicity, capturing many of their key properties, such as  associativity formulas and additivity over short exact sequences.  It is well known that while Hilbert-Kunz multiplicity and Hilbert-Samuel multiplicity are upper semicontinuous, the $h$-function and $s$-multiplicity are  Lipschitz continuous function of $s$.

Meng and  Mukhopadhyay \cite{MM} employed convex-geometric methods to extend the theory of $h$-functions and thereby the theories of Frobenius-Poincar$\acute{e}$ functions, and Hilbert-Kunz density functions, from the graded setting to the general local case, thereby resolving a question posed by Trivedi.

It is well known that both the Hilbert-Kunz multiplicity and the Hilbert-Samuel multiplicity are always greater than or equal to one. Motivated by this fact, Taylor and Miller posed the following question in (\cite[Question 2.9]{MT}):

 \vspace{2mm} \noindent{\bf Question:} {\it{Is $e_{s}(R )\geq 1$ for all local rings $R$ of  positive characteristic$?$}}

Although they resolved the Cohen-Macaulay case affirmatively, the non-Cohen-Macaulay case remains open. Previous work shows that under certain conditions, Cohen-Macaulay rings can give rise to non-Cohen-Macaulay constructions: specifically,\cite[Proposition 1.7]{AAM} establishes that when $\text{dim}(R)>\text{dim}(S)>\text{dim}(T)$ with $R $ Cohen-Macaulay, the fiber product $R \times_T S$ is non-Cohen-Macaulay, while \cite[Corollary 4.14]{DM} shows the same for the idealization $R \ltimes M$. In Sections 3 and 4, we developed formulas for the h-function and s-multiplicity in these two settings. Leveraging these results, we now generalize Taylor and Miller's theorem to both cases (see Theorems   \ref{thm:5.1} and \ref{thm:5.2}(3)):

\begin{thm}\label{thm:1.1}{\it{Let  $(R, \mathfrak{m})$  be a Cohen-Macaulay   local ring of prime characterstic $p>0$, and let $M $ be a  finitely generated $R$-module.  Then  $$ e_{s}(\mathfrak{m} \ltimes \mathfrak{m}M)\geq 1.$$}} \end{thm}

\begin{thm}\label{thm:1.2}{\it{Let $(R, \mathfrak{m}_R, k)$, $(S, \mathfrak{m}_S, k)$ and $(T, \mathfrak{m}_T, k)$ be   local rings of prime characteristic $p>0$. If $\text{dim}(R)>\text{dim}(S)>\text{dim}(T)$ and $R $ is Cohen-Macaulay, then    $$e_{s}(R \times_T S)\geq 1.$$
}} \end{thm}

 Our theorem thus provides a significant extension by demonstrating that the Taylor-Miller question admits a positive answer for these broader classes of non-Cohen-Macaulay rings.

\textbf{Watanabe-Yoshida  conjecture} (\cite{WY}): Let $(R, \mathfrak{m},k)$ be a  local ring of dimension $d\geq1$ and prime characteristic $p>2$, where  $k$ is  an algebraically closed field. Define $R_{d}=k[[x_{0},\cdots,x_{d} ]]/(x_{0}^{2}+\cdots+x_{d}^{2})$ for $d\geq1$.  Watanabe and Yoshida proposed the following conjecture regarding inequalities for the Hilbert-Kunz multiplicity:

$\mathrm{(1)}$ If $R$ is non-regular, then $$e_{HK}(R)\geq e_{HK}(R_{d})\geq1+\frac{c_{d}}{d!},$$ where the numbers $\frac{c_{d}}{d!}$ satisfies $sec(x)+tan(x)=1+\mathop{\sum}\limits_{ d=1 }^{\infty}\frac{c_{d}}{d!}x^{d},(-\pi/2<x<\pi/2)$.

  $\mathrm{(2)}$ If   $e_{HK}(R)= e_{HK}(R_{d})$,  then the
$\mathfrak{m}$-adic completion $\widehat{R}$ of $R$  satisfies $\widehat{R^{\mathfrak{m}}}\cong  R_{d}$  as local rings.

The Watanabe-Yoshida  conjecture has been partially verified under various  conditions. Recently,  Meng \cite[Theorems 7.2 and  7.8]{M} proved the Watanabe-Yoshida  conjecture for all odd primes through the application of  $h$-function theory. Independently,   Castillo-Rey \cite[Theorems A and  C]{R} proved the   strong form of the Watanabe-Yoshida  conjecture for complete intersection singularities in every positive characteristic. The conjecture has inspired significant progress, yet a complete resolution remains an open problem.

Taylor and Miller \cite[section 3]{MT}, who investigated the  s-analogue of Watanabe-Yoshida  conjecture for $s$-multiplicity, they proved  $ e_{s}(R) \geq e_{s}(R_{d}) $ in unmixed non-regular ring of dimension at most $3$.

 Based on  results in Sections 3 and 4 and on work of Taylor and Miller \cite[section 3]{MT}, we generalize their theorems to these two settings. This approach also reveals new connections between the Watanabe-Yoshida conjecture and the behavior of s-multiplicity (see Theorems   \ref{thm:5.4} and \ref{thm:5.6}):

\begin{thm}\label{thm:1.3}{\it{Let  $d\geq 2$ and  let $k$ be a field of prime characterstic $p>2$. Suppose   $(R, \mathfrak{m}_R, k)$, $(S, \mathfrak{m}_S, k)$ and $(T, \mathfrak{m}_T, k)$   are  local rings of prime characterstic $p>0$, and assume $R$  is a non-regular
 complete intersection  of dimension  $d$. If $\text{dim}(R)>\text{dim}(S)>\text{dim}(T)$, then for any  $s>0 $, we have  $$ e_{s}(R \times_T S) \geq e_{s}(R_{d}).$$}} \end{thm}

\begin{thm}\label{thm:1.4}{\it{Let  $d\geq 2$ and let $k$ be a field of prime characterstic $p>2$. Let  $(R, \mathfrak{m},k)$  be a
non-regular complete intersection  of dimension  $d$ with  prime characterstic $p>0$ and  $M $  a finitely generated $R$-module. Then for any  $s>0 $, we have  $$ e_{s}(\mathfrak{m} \ltimes \mathfrak{m}R) \geq e_{s}(R_{d}).$$}} \end{thm}

\bigskip
\section{\bf  Preliminaries}

Throughout this paper, all rings are assumed to be commutative Noetherian local rings with identity, unless otherwise specified. When a local ring $(R, \mathfrak{m},k )$ has prime characteristic $p > 0$, the residue field $k$ also has characteristic $p > 0$, and we assume the dimension $d := \text{dim}R $  is positive. In particular,  we denote the length  of an $R$-module by $\ell_{R}(-)$.

In this section, we offer a concise overview of the theories of Hilbert-Kunz multiplicity, Hilbert-Samuel multiplicity,
$h$-function, and $s$-multiplicity, along with the notions of fiber products of rings and idealizations of module. Key results to be used in later proofs are also presented.

\vspace{2mm} \noindent{\bf The Hilbert-Kunz multiplicity and Hilbert-Samuel multiplicity}
\mbox{}\\

Let  $(R, \mathfrak{m} )$ be a  local ring,  let $I$ be an $\mathfrak{m}$-primary ideal of  $R$, and let $M $  be a finitely generated $R$-module.

$\mathrm{(1)}$ The Hilbert-Samuel multiplicity of  $M $ with respect to $I$ is defined as:
$$e(I,M)
=\mathop{\text{lim}}\limits_{ n\rightarrow \infty }\frac{d!  \ell_{R}( M/I^{n}M)}{n^{d}},$$
where $I^{n}$ is the ordinary power of $I$.   For simplicity, we write $e(I,R)$ as $e(I)$, $e(R)$ as $e(\mathfrak{m},R)$.

$\mathrm{(2)}$ If $R$ has prime characteristic $p>0$, the Hilbert-Kunz multiplicity  of $M $ with respect to $I$ is defined as:
$$e_{HK}(I,M)
=\mathop{\text{lim}}\limits_{ e\rightarrow \infty }\frac{\ell_{R}( M/I^{[p^{e}]}M)}{p^{ed}},$$where $I^{[p^{e}]}$  denotes the ideal generated by the $p^{e}$-th powers of elements of  $I$.   For simplicity, we write $e_{HK}(I,R)$ as $e_{HK}(I)$, $e_{HK}(R)$ as $e_{HK}(\mathfrak{m},R)$.

\vspace{2mm} \noindent{\bf Fiber product of rings}
\mbox{}\\

Let $(R, \mathfrak{m}_R )$, $(S, \mathfrak{m}_S)$ and $(T, \mathfrak{m}_T)$ be  local rings  sharing the same residue field $k$,  and let $R \xrightarrow{\epsilon_{R}} T \xleftarrow{\epsilon_{S}} S$ be surjective ring homomorphisms. The fiber product of $R$ and $S$ over $T$ is defined as:
 $R \times_T S = \{(r,s) \in R \times S \mid \epsilon_{R}(r) = \epsilon_{S}(s)\}$ is a  local ring with maximal ideal $\mathfrak{m} = \mathfrak{m}_R \times_{\mathfrak{m}_T} \mathfrak{m}_S$ and residue field $k$.  It is a subring of the direct product $R \times S$. Let $\eta_R: R \times_T S \twoheadrightarrow R$ and $\eta_S: R \times_T S \twoheadrightarrow S$ be the natural projections $(r,s) \mapsto r$ and $(r,s) \mapsto s$, respectively. Then ring $R \times_T S$ can be represented by the following pullback diagram:
$$\xymatrix@C=2cm@R=0.7cm{
  R \times_T S \ar[d]^{\eta_R} \ar[r]^{\eta_S}
                &S\ar[d]_{\epsilon_{S}} \\
  R \ar[r]^{\epsilon_{R}}
                &T            \\
                            }$$
 Note that every (finitely generated) module over $R$, $S$ or $T$  naturally inherits a (finitely generated) module structure over $R \times_T S$.  In particular, when $R$ and $S$ share the same prime characteristic, so does  fiber product  $R \times_T S$. For further details, we refer to  (see \cite{AAM}).
\begin{lem}\label{lem:2.1}{\it{For the fiber product $R \times_T S$,  the following properties hold:

$\mathrm{(1)}$ The sequence of $R \times_T S$-module   $0 \longrightarrow R \times_T S \xrightarrow{\ \ } R \oplus S \xrightarrow{\  \ } T \longrightarrow 0$ is exact.

$\mathrm{(2)}$ The following dimension (in)equality holds: $$\text{dim}(R \times_TS)=\text{max}\{\text{dim}(S),\text{dim}(R)\}\geq\text{min}\{\text{dim}(S),\text{dim}(R)\}\geq \text{dim}(T).$$
}} \end{lem}

\vspace{2mm} \noindent{\bf Idealization of a module}
\mbox{}\\

Let  $R $ be a commutative ring and $M $ an  $R$-module.  The idealization of $M $ over $R$, also known as the trivial ring extension, is the  ring  $ R\ltimes M=\{(r,m ),r\in R, m\in M\}$. This ring is commutative with identity element$(1,0)$.
\begin{lem}\label{lem:2.2}{\it{(see \cite[Chapter VI]{H}) Let  $(R, \mathfrak{m} )$ be a  local ring of prime characteristic $p>0$ and let $M $ be a finitely generated $R$-module. Then

$\mathrm{(1)}$ The sequence  $0 \longrightarrow R  \xrightarrow{ } R\ltimes M  \xrightarrow{ } M  \longrightarrow 0$ is exact.

$\mathrm{(2)}$ The ring $R \ltimes M$  is also a  local ring of prime characteristic $p>0$  with maximal ideal $\mathfrak{m} \ltimes R$ and residue field    $k$, In particular, $\text{dim}(R \ltimes M)=\text{dim}(R)$.

$\mathrm{(3)}$ If $N$  is an  $R \ltimes M$-module, then $\ell_{R}(N)= \ell_{R \ltimes M}(N)$.
}} \end{lem}

\vspace{2mm} \noindent{\bf The $h$-function and $s$-multiplicity}
\mbox{}\\

In this subsection, we begin by reviewing definition and some fundamental properties  of the $h$-function and $s$-multiplicity.

 Let  $(R, \mathfrak{m} )$ be a   local ring of prime characterstic $p>0$,  let $I$ and $J$ be $\mathfrak{m}$-primary ideals of  $R$, and let $M $ be a finitely generated $R$-module. For a real number $s> 0$,

$\mathrm{(1)}$  Define the $h$-function of $M$ with respect to the pair $(I, J)$ as:
$$h_{s}(I,J;M)
=\mathop{\text{lim}}\limits_{ e\rightarrow \infty }\frac{\ell_{R}(M/(I^{\lceil sp^{e}\rceil}+J^{[p^{e}]})M)}{p^{ed}},$$
 where $J^{[p^{e}]}$  denotes the ideal generated by the $p^{e}$-th powers of elements of  $J$,   $I^{\lceil sp^{e}\rceil}$ is the ordinary power of $I$ with exponent given by the ceiling of $sp^{e}$.  For simplicity, we adopt the following abbreviated notation $h_{s}(I,J):=h_{s}(I,J;R)$, $h_{s}(I;M):=h_{s}(I,I;M)$, $h_{s}(I):=h_{s}(I;R)$ and $h_{s}(M):=h_{s}(\mathfrak{m};M)$. If we wish to emphasize the underlying ring, we may write it explicitly as $h_{s}^{R}(I,J;M)$. In particular, if  $R$ is regular, then $h_{s}(R)=\mathop{\sum}\limits_{ i=0 }^{\lfloor s\rfloor}\frac{(-1)^{i}}{d!}\left(   d \atop i \right)(s-i)^{d}$, which serves as a normalizing factor, and one defines the
$s$-multiplicity as follows.

$\mathrm{(2)}$  Define the $s$-multiplicity of $M$ with respect to the pair $(I, J)$ as:
$$e_{s}(I,J;M)
=\frac{h_{s}(I,J;M)}{\mathcal{H}_{s}(d)},$$
where $\mathcal{H}_{s}(d)=\mathop{\sum}\limits_{ i=0 }^{\lfloor s\rfloor}\frac{(-1)^{i}}{d!}\left(   d \atop i \right)(s-i)^{d}$, and  $\lfloor\cdot \rfloor$  denotes the floor function. Note that $\mathcal{H}_{s}(d)$ depends only on $s$ and the dimension $d:=\text{dim}(R)$. For simplicity, we adopt the following abbreviated notation $e_{s}(I,J):=e_{s}(I,J;R)$, $e_{s}(I;M):=e_{s}(I,I;M)$, $e_{s}(I):=e_{s}(I;R)$ and $e_{s}(M):=e_{s}(\mathfrak{m};M)$.

\begin{lem}\label{lem:2.3}{\it{(\cite{T}and \cite{MT}) Let  $(R, \mathfrak{m} )$ be a local ring of prime characteristic $p>0$. Let  $I$ and $J$ be $\mathfrak{m}$-primary ideals of  $R$, and let $M $ be a finitely generated $R$-module. Then the following results hold:

$\mathrm{(1)}$  If $\text{dim} M < \text{dim} R$, then $h_{s}(I,J;M) = 0$.

$\mathrm{(2)}$ If $0 \longrightarrow M \xrightarrow{\\ } N \xrightarrow{\  \ } L \longrightarrow 0$ is a short exact sequence of finitely generated $R$-modules, then $ h_{s}(I,J;N)= h_{s}(I,J;M)+h_{s}(I,J;L)$.

$\mathrm{(3)}$   If $\text{dim} M < \text{dim} R$, then $e_{s}(I,J;M) = 0.$

$\mathrm{(4)}$ If $R$ is Cohen-Macaulay, then $ e_{s}(R)\geq 1 $.
}} \end{lem}

\bigskip
\section{\bf The $h$-function and the $s$-multiplicity  of  fiber product rings}
In this section, we derive some formulas for the
$h$-function and the $s$-multiplicity of a fiber product ring. Specifically, we investigate relationships between the $h$-function of a fiber product $R \times_T S$  and  the $h$-functions of $R$, $ T$  and $ S$, regarded as modules over $R \times_T S$.  We begin with the following main theorem in this section.

\begin{thm}\label{thm:3.1}{\it{Let $(R, \mathfrak{m}_R, k)$, $(S, \mathfrak{m}_S, k)$ and $(T, \mathfrak{m}_T, k)$ be   local rings, where $R$ and $S$ share the same  prime characteristic $p>0$, and let $I$ and  $J$ be $\mathfrak{m}$-primary ideals of  $R \times_T S$. Then the following holds: $$ h_{s}(I,J;R \times_T S)=\left\{
\begin{array}{lcl}
h_{s}(I,J;R )    +h_{s}(I,J; S) -h_{s}(I,J;T );  &       {\text{dim}(R)=\text{dim}(S)=\text{dim}(T)}\\
h_{s}(I,J;R )    +h_{s}(I,J; S);   &       {\text{dim}(R)=\text{dim}(S)>\text{dim}(T)}\\
h_{s}(I,J;R ) ;    &       {\text{dim}(R)>\text{dim}(S)>\text{dim}(T)}.
\end{array} \right. $$}} \end{thm}
\begin{proof}As both
$R$ and $S$ are of prime characteristic $p>0$, so too is their fiber product $R \times_T S$. By Lemma \ref{lem:2.1}(1), we have the following short exact sequence of finitely generated $R \times_T S$-modules
$$0 \longrightarrow R \times_T S \xrightarrow{\ \ } R \oplus S \xrightarrow{\  \ } T \longrightarrow 0,$$ by the additivity of the $h$-function in the short exact sequences in Lemma \ref{lem:2.3}(2), we get
 $$ h_{s}(I,J;R \times_T S)=h_{s}(I,J;R \oplus S ) -h_{s}(I,J;T ).$$Applying the same argument to the short exact sequence of finitely generated $R \times_T S$-modules$$0 \longrightarrow R \xrightarrow{ } R \oplus S \xrightarrow{ } S \longrightarrow 0,$$ we get
 $$ h_{s}(I,J;R \oplus S)=h_{s}(I,J;R   )+h_{s}(I,J;S ).$$By combining these equalities, we obtain that
$$ h_{s}(I,J;R \times_T S)=h_{s}(I,J;R   )+h_{s}(I,J;S ) -h_{s}(I,J;T ).$$ Now using Lemma \ref{lem:2.1}(2), we obtain the following dimension (in)equalities
 $$\text{dim}(R \times_T S)=\text{max}\{\text{dim}(S),\text{dim}(R)\}\geq\text{min}\{\text{dim}(S),\text{dim}(R)\}\geq \text{dim}(T).$$ By  Lemma \ref{lem:2.3}(1), if     $\text{dim}(M)< \text{dim}(R \times_T S) $ for any finitely generated $R \times_T S$-module $M$,  then $h_{s}(I,J;M   )=0$. The following proves three situations:

 If   $\text{dim}(R)=\text{dim}(S)=\text{dim}(T)$, then $\text{dim}(R \times_T S)=\text{dim}(R)=\text{dim}(S)=\text{dim}(T)$, and therefore $$ h_{s}(I,J;R \times_T S)=h_{s}(I,J;R   )+h_{s}(I,J;S ) -h _{s}(I,J;T ).$$

  If   $\text{dim}(R)=\text{dim}(S)>\text{dim}(T)$, then $\text{dim}(R \times_T S)=\text{dim}(R)=\text{dim}(S)>\text{dim}(T)$, which implies $h_{s}(I,J;T )=0$. Therefore $$ h_{s}(I,J;R \times_T S)=h_{s}(I,J;R   )+h_{s}(I,J;S ).$$

 If   $\text{dim}(R)>\text{dim}(S)>\text{dim}(T)$, then $\text{dim}(R \times_T S)=\text{dim}(R)>\text{dim}(S)>\text{dim}(T)$, which implies $h_{s}(I,J;T )=0$ and $h_{s}(I,J;S )=0$. Hence, we obtain $$ h_{s}(I,J;R \times_T S)=h_{s}(I,J;R   ).$$This completes the proof.
\end{proof}

\begin{cor}\label{cor:3.2}{\it{Let $(R, \mathfrak{m}_R, k)$, $(S, \mathfrak{m}_S, k)$ and $(T, \mathfrak{m}_T, k)$ be   local rings, where $R$ and $S$ share the same  prime characteristic $p>0$, and let $I$ and  $J$ be $\mathfrak{m}$-primary ideals of  $R \times_T S$. Then the following holds: $$ e_{s}(I,J;R \times_T S)=\left\{
\begin{array}{lcl}
e_{s}(I,J;R)    +e_{s}(I,J; S) -e_{s}(I,J;T );  &       {\text{dim}(R)=\text{dim}(S)=\text{dim}(T)}\\
e_{s}(I,J;R)    +e_{s}(I,J; S);   &       {\text{dim}(R)=\text{dim}(S)>\text{dim}(T)}\\
e_{s}(I,J;R);     &       {\text{dim}(R)>\text{dim}(S)>\text{dim}(T)}.
\end{array} \right. $$}} \end{cor}
\begin{proof}The value  $\mathcal{H}_{s}(d)$ depends solely on the dimension $d=\text{dim}(R \times_T S) $  and the parameter $s$. The desired formulas follow immediately from the dimension formula in Lemma \ref{lem:2.1}(2) and Theorem \ref{thm:3.1}.
\end{proof}

 Next, we explore relationships between the $h$-function of  the maximal ideal $\mathfrak{m}$ of a fiber product $R \times_T S$ and $h$-functions  of maximal ideal $\mathfrak{m}_{R}, \mathfrak{m}_{S}$ and
 $\mathfrak{m}_{T}$ of the $R$, $ T$ and $ S$, respectively. To this end, we first prove a key proposition that will be essential in the subsequent argument.

\begin{prop}\label{prop:3.3}{\it{Let  $f:(  A,\mathfrak{n}_{A},k)\rightarrow (B,\mathfrak{n}_{B},k)$ be a surjctive homomorphism of
  local rings of prime characteristic $p>0$.  Then  $h_{s}^{A}(B)=h_{s}^{B}(B)$.}} \end{prop}
\begin{proof}Since   $f$ is a ring homomorphism, $f(\mathfrak{n}_{A})$  is an ideal of $B$. Since $f$ is surjective, $ f(\mathfrak{n}_{A})$ is a proper ideal in $ B $. If $ f(\mathfrak{n}_{A}) = B $, then there exists $ a \in \mathfrak{n}_{A} $ such that $ f(a) = 1_B $, which contradicts the fact that $ \mathfrak{n}_{A}$ is a maximal ideal. Since $\mathfrak{n}_{A} $ is the maximal ideal of $ A $ and $ f $ is surjective, $ f(\mathfrak{n}_{A}) $ must be the maximal ideal of $ B $. Because $ B $ is a local ring, it has only one maximal ideal $ \mathfrak{m}_B $. Hence $f(\mathfrak{n}_{A})=\mathfrak{n}_{B}$. It follows that for any positive integer $e$, $\mathfrak{n}_{A}^{[p^{e}]}B=\mathfrak{n}_{B}^{[p^{e}]}$ and thus for all $s>0$,
$\mathfrak{n}_{A}^{\lceil sp^{e}\rceil}B=\mathfrak{n}_{B}^{\lceil sp^{e}\rceil}$. By assumption, $A$ and $B$  have the same residue field $k$, and every
 $B$-module $M$ may be viewed as an $A$-module via $f$.  Therefore, for any $B$-module $M$, $\ell_{B}(M)=\ell_{A}(M)$. Consequently,
\begin{center}
$\begin{aligned}
h_{s}^{A}(B)
&=\mathop{\text{lim}}\limits_{ e\rightarrow \infty }\frac{\ell_{A}(B/(\mathfrak{n}_{A}^{\lceil sp^{e}\rceil}+\mathfrak{n}_{A}^{[p^{e}]})B)}{p^{ed}}\\
&=\mathop{\text{lim}}\limits_{ e\rightarrow \infty }\frac{\ell_{B}(B/(\mathfrak{n}_{B}^{\lceil sp^{e}\rceil}+\mathfrak{n}_{B}^{[p^{e}]}))}{p^{ed}}\\
&=h_{s}^{B}(B).
\end{aligned}$
\end{center}This completes the proof.
\end{proof}

 We now establish formulas   relating the $h$-function of the fiber product $R \times_T S$ in terms of  the $h$-functions  of $R$, $ T$ and $S$.
\begin{lem}\label{lem:3.4}{\it{Let $(R, \mathfrak{m}_R, k)$, $(S, \mathfrak{m}_S, k)$ and $(T, \mathfrak{m}_T, k)$ be   local rings of prime characteristic $p>0$. Then the following properties hold for the fiber product  $R \times_TS$: $$ h_{s}(R \times_T S)=\left\{
\begin{array}{lcl}
h_{s}(R )    +h_{s}( S) -h_{s}(T );  &       {\text{dim}(R)=\text{dim}(S)=\text{dim}(T)}\\
h_{s}(R )    +h_{s}( S);   &       {\text{dim}(R)=\text{dim}(S)>\text{dim}(T)}\\
h_{s}(R ) ;    &       {\text{dim}(R)>\text{dim}(S)>\text{dim}(T)}.
\end{array} \right. $$}} \end{lem}
\begin{proof}Since $R$ and $S$ both have prime characteristic $p>0$, the fiber product  $R \times_T S$  also has characteristic $p>0$. Therefore, by Theorem \ref{thm:3.1} and   taking  $I=J=\mathfrak{m}$, we obtain   $$ h_{s}(R \times_T S)=\left\{
\begin{array}{lcl}
h_{s}^{R \times_T S}(R )    +h_{s}^{R \times_T S}( S) -h_{s}^{R \times_T S}(T );  &       {\text{dim}(R)=\text{dim}(S)=\text{dim}(T)}\\
h_{s}^{R \times_T S}(R )    +h_{s}^{R \times_T S}( S);   &       {\text{dim}(R)=\text{dim}(S)>\text{dim}(T)}\\
h_{s}^{R \times_T S}(R );    &       {\text{dim}(R)>\text{dim}(S)>\text{dim}(T)}.
\end{array} \right. $$Since $\eta_R: R \times_T S \twoheadrightarrow R$ and $\eta_S: R \times_T S \twoheadrightarrow S$ are the natural projections defined by $(r,s) \mapsto r$ and $(r,s) \mapsto s$,   respectively, it follows from Proposition \ref{prop:3.3} that, $$h_{s}^{R \times_T S}(R)
=h_{s}^{R}(R),\quad h_{s}^{R \times_T S}(S)=h_{s}^{S}(S).$$
 Given that $R \xrightarrow{\epsilon_{R}} T \xleftarrow{\epsilon_{S}} S$ are surjective ring homomorphisms, the induced map  $R \times_T S \rightarrow T $ is also a surjective ring homomorphism. Therefore, again by Proposition \ref{prop:3.3}, we have $h_{s}^{R \times_T S}(T)=h_{s}^{T}(T)$. This completes the proof.
\end{proof}

The following corollary provide formulas relating the $s$-multiplicity of the fiber product
$R \times_T S$  in terms of the corresponding  $s$-multiplicities of $R$, $ T$ and $ S$.

\begin{cor}\label{cor:3.5}{\it{Let $(R, \mathfrak{m}_R, k)$, $(S, \mathfrak{m}_S, k)$ and $(T, \mathfrak{m}_T, k)$ be   local rings of prime characteristic $p>0$. Then the following properties hold for the fiber product  $R \times_TS$: $$ e_{s}(R \times_T S)=\left\{
\begin{array}{lcl}
e_{s}(R )    +e_{s}( S) -e_{s}(T );  &       {\text{dim}(R)=\text{dim}(S)=\text{dim}(T)}\\
e_{s}(R )    +e_{s}( S);   &       {\text{dim}(R)=\text{dim}(S)>\text{dim}(T)}\\
e_{s}(R );     &       {\text{dim}(R)>\text{dim}(S)>\text{dim}(T)}.
\end{array} \right. $$}} \end{cor}
\begin{proof}The value  $\mathcal{H}_{s}(d)$ depends solely on the dimension $d=\text{dim}(R \times_T S) $  and the parameter $s$. The desired formulas follow immediately from the dimension formula in Lemmas \ref{lem:2.1}(2) and  \ref{lem:3.4}.
\end{proof}

Next, we explore the relationship between the $h$-function and the $s$-multiplicity of $R$ and those of the amalgamated duplication $R \bowtie I$ of $R$ along the ideal $I$. We begin by recalling the definition of the amalgamated duplication.

Let $(R, \mathfrak{m} )$ be a  local ring.  Consider the canonical maps $ \pi :
R \rightarrow R/I$  and the identity map $ \iota :
R \rightarrow R$ , where
$I$ is a proper ideal of  $R$. Then there is an isomorphism
 $$R \times_{R/I} R \cong R \bowtie I,$$ where $R \bowtie I= \{(r,s)   \mid r,s \in R, s-r\in I\}$ denotes the amalgamated duplication of $R$ along $I$,  introduced by D$^{\prime}$Anna \cite{D}.  In particular, if $R$ is local with maximal ideal $\mathfrak{m}$ and dimension $d$, then
$R \bowtie I$ is also a  local ring with maximal ideal $\mathfrak{m} \bowtie I$ and dimension $d$.

\begin{cor}\label{cor:3.6}{\it{Let  $(R, \mathfrak{m} )$ be a   local ring of prime characteristic $p>0$, and let $I$ be  a proper ideal of  $R $ such that $\text{dim}(R/I)<\text{dim}(R)$. Then  $$ h_{s}(R \bowtie I)=2h_{s}(R).$$}} \end{cor}
\begin{proof}Taking $ R = S $ and $
T= R/I$, then $\text{dim}(R)=\text{dim}(S)>\text{dim}(R/I)$, and by Lemma  \ref{lem:3.4}, we obtain $ h_{s}(R \bowtie I)=2h_{s}(R ) $.
\end{proof}

\begin{cor}\label{cor:3.7}{\it{Let  $(R, \mathfrak{m} )$ be a   local ring of prime characteristic $p>0$, and let $I$ be  a proper  ideal of  $R $ such that $\text{dim}(R/I)<\text{dim}(R)$. Then  $$ e_{s}(R \bowtie I)=2e_{s}(R).$$}} \end{cor}
\begin{proof}Since $\text{dim}(R)=\text{dim}(R \bowtie I)$ and  value $\mathcal{H}_{s}(d)$ depends only on the dimension $d=\text{dim}(R)$ and the parameter $s$, the result is a direct consequence of Corollary \ref{cor:3.6}.
\end{proof}

\bigskip
\section{\bf The $h$-function and $s$-multiplicity of   idealization
rings}
In this section, we establish fundamental properties of the $h$-function and $s$-multiplicity  for idealization rings. We begin by investigating relationships between the $h$-function of idealization
ring $R \ltimes M$ and the $h$-functions of  $R$ and $ M$. To this end, we first prove a key lemma that will be essential in the subsequent proofs.
\begin{lem}\label{lem:4.1}{\it{Let  $(R, \mathfrak{m})$ be a   local ring   and  $M $  a finitely generated $R$-module.  If $I$  is an ideal of  $R $ and $J=I \ltimes IM$, then  $J^{n}= I^{n}(R \ltimes M)$ for any positive integer $n $.   }} \end{lem}
\begin{proof}Since $J=I\ltimes IM$, we have
 $J^{n}=(I\ltimes IM)^{n}=I^{n}\ltimes I^{n}M=I^{n}(R \ltimes M)$ by \cite[Theorem 25.1]{H}(4).
\end{proof}

\begin{thm}\label{thm:4.2}{\it{Let  $(R, \mathfrak{m})$  be a   local ring of prime characteristic $p>0$, and  let $M $ be a finitely generated $R$-module. Suppose $I_{1}$ and $I_{2}$ are  $\mathfrak{m}$-primary ideals of  $R $ and $J_{1}=I_{1}\ltimes I_{1}M$, $J_{2}=I_{2}\ltimes I_{2}M$. Then the following properties hold:
\begin{center}
$\begin{aligned}
h_{s}(J_{1},J_{2};R \ltimes M)
&=h_{s}(I_{1},I_{2};R ) +h_{s}(I_{1},I_{2};M )\\
&=\mathop{\text{lim}}\limits_{ e\rightarrow \infty }\frac{\ell_{R}(R\ltimes M/((I_{1}^{\lceil sp^{e}\rceil}+I_{2}^{[p^{e}]})\ltimes (I_{1}^{\lceil sp^{e}\rceil}+I_{2}^{[p^{e}]}) M))}{p^{ed}}.
\end{aligned}$
\end{center}}} \end{thm}
\begin{proof}By assumption, $I_{1}$ and $I_{2}$ are $\mathfrak{m}$-primary ideals of  $R $.  Then $J_{1}=I_{1}\ltimes I_{1}M$ and $J_{2}=I_{2}\ltimes I_{2}M$ are  $\mathfrak{m} \ltimes M$-primary ideals of  $R\ltimes M $ by \cite[Theorem 25.2]{H}. By Lemma \ref{lem:2.3}(1), there is a  short exact sequence of $R $-modules

$$0 \longrightarrow R  \xrightarrow{ } R\ltimes M  \xrightarrow{ } M  \longrightarrow 0,$$
 by the additivity of the $h$-function in the short exact sequences in Lemma \ref{lem:2.3}(2), we get
 $$ h_{s}(I_{1},I_{2};R\ltimes M)=h_{s}(I_{1},I_{2};R ) +h_{s}(I_{1},I_{2};M ).$$
 By Lemma \ref{lem:2.3}(3), $l_{R}(N) = l_{R\ltimes M}(N)$ for each $R\ltimes M$-module $N$. Using \cite[Lemma 4.1]{VPJ}, we have
$J_{2}^{[p^{e}]} = I_{2}^{[p^{e}]}(R\ltimes M)$ and by Lemma \ref{lem:4.1}, we obtain $J_{1}^{\lceil sp^{e}\rceil} = I_{1}^{\lceil sp^{e}\rceil}(R\ltimes M)$. Then $J_{2}^{[p^{e}]}+ J_{1}^{\lceil sp^{e}\rceil}= (I_{2}^{[p^{e}]}+I_{1}^{\lceil sp^{e}\rceil})(R\ltimes M)$. Therefore, we obtain
 \begin{center}
$\begin{aligned}
h_{s}(J_{1},J_{2};R\ltimes M)
&=\mathop{\text{lim}}\limits_{ e\rightarrow \infty }\frac{\ell_{R\ltimes M}(R\ltimes M/(J_{1}^{\lceil sp^{e}\rceil}+J_{2}^{[p^{e}]}))}{p^{ed}}\\
&=\mathop{\text{lim}}\limits_{ e\rightarrow \infty }\frac{\ell_{R}(R\ltimes M/(I_{1}^{\lceil sp^{e}\rceil}+I_{2}^{[p^{e}]})R\ltimes M)}{p^{ed}}\\
&=h_{s}(I_{1},I_{2};R\ltimes M)\\
&=h_{s}(I_{1},I_{2};R ) +h_{s}(I_{1},I_{2};M ).
\end{aligned}$
\end{center}
 By  \cite[Theorem 3.1]{DM}, there is an isomorphism $$\frac{R\ltimes M}{(I_{1}^{\lceil sp^{e}\rceil}+I_{2}^{[p^{e}]})\ltimes (I_{1}^{\lceil sp^{e}\rceil}+I_{2}^{[p^{e}]}) M}\cong\frac{R}{(I_{1}^{\lceil sp^{e}\rceil}+I_{2}^{[p^{e}]})} \ltimes \frac{M}{(I_{1}^{\lceil sp^{e}\rceil}+I_{2}^{[p^{e}]})M},$$
 again, by \cite[Theorem 3.1]{DM} and Lemma \ref{lem:2.3}(1), we have the following short exact sequence$$0 \longrightarrow \frac{R}{(I_{1}^{\lceil sp^{e}\rceil}+I_{2}^{[p^{e}]})}  \xrightarrow{ } \frac{R\ltimes M}{(I_{1}^{\lceil sp^{e}\rceil}+I_{2}^{[p^{e}]})\ltimes (I_{1}^{\lceil sp^{e}\rceil}+I_{2}^{[p^{e}]}) M}  \xrightarrow{ } \frac{M}{(I_{1}^{\lceil sp^{e}\rceil}+I_{2}^{[p^{e}]})M}  \longrightarrow 0.$$
 By length additivity in the short exact sequences,  we get
 $$\ell_{R}(\frac{R\ltimes M}{(I_{1}^{\lceil sp^{e}\rceil}+I_{2}^{[p^{e}]})\ltimes (I_{1}^{\lceil sp^{e}\rceil}+I_{2}^{[p^{e}]}) M})= \ell_{R}(\frac{R}{(I_{1}^{\lceil sp^{e}\rceil}+I_{2}^{[p^{e}]})})+\ell_{R}(\frac{M}{(I_{1}^{\lceil sp^{e}\rceil}+I_{2}^{[p^{e}]})M} ).$$
 Therefore, dividing both sides by $ p^{ed}$ and taking the limit as $e$ goes to infinity gives us
 $$ \mathop{\text{lim}}\limits_{ e\rightarrow \infty }\frac{\ell_{R}(R\ltimes M/((I_{1}^{\lceil sp^{e}\rceil}+I_{2}^{[p^{e}]})\ltimes (I_{1}^{\lceil sp^{e}\rceil}+I_{2}^{[p^{e}]}) M))}{p^{ed}}=h_{s}(I_{1},I_{2};R ) +h_{s}(I_{1},I_{2};M ).$$This completes the proof.
\end{proof}

\begin{cor}\label{cor:4.3}{\it{Let  $(R, \mathfrak{m})$  be a   local ring of prime characteristic $p>0$, and let $M $ be a finitely generated $R$-module. Then
$$h_{s}(R \ltimes M)\leq h_{s}(\mathfrak{m} \ltimes \mathfrak{m}R)=h_{s}(R ) +h_{s}(M ).$$ In particular, if $\text{dim}(M)<\text{dim}(R)$, then $$h_{s}(R \ltimes M)\leq h_{s}(\mathfrak{m} \ltimes \mathfrak{m}R)= h_{s}(R ).$$
}} \end{cor}
\begin{proof} By \cite[Proposition 2.6]{T}, the first inequality follows immediately from the inclusion   $\mathfrak{m} \ltimes \mathfrak{m}R\subseteq\mathfrak{m} \ltimes R$.  Now, taking
$I_{1}=I_{2}=\mathfrak{m}$, the result follows directly from Theorem \ref{thm:4.2}. This completes the proof.
\end{proof}

\begin{cor}\label{cor:4.4}{\it{Let  $(R, \mathfrak{m})$  be a   local ring of prime characteristic $p>0$ and  $M $  a finitely generated $R$-module. Then
$$e_{s}(R \ltimes M)\leq e_{s}(\mathfrak{m} \ltimes \mathfrak{m}R)=e_{s}(R ) +e_{s}(M ).$$ In particular, if $\text{dim}(M)<\text{dim}(R)$, then $$e_{s}(R \ltimes M)\leq e_{s}(R ) .$$}} \end{cor}
\begin{proof}
 It follows immediately from Lemma \ref{lem:2.3}(2) that $\text{dim}(R \ltimes M)=\text{dim}(R)$, and since  value $\mathcal{H}_{s}(d)$ depends only on the dimension $d=\text{dim}(R)$ and the parameter $s$,
 the result is a direct consequence of Corollary \ref{cor:4.3}.
\end{proof}

 Next, we investigate the relationships between the $h$-function and several numerical invariants of $M$   in the idealization ring $R \ltimes M$. To this end, we first prove a lemma that will be essential in the proof. Denote $ \text{Assh}R=\{\mathfrak{p}  \in \text{Spec}(R)\mid \text{dim}R=\text{dim}R/\mathfrak{p}\}$.

\begin{lem}\label{lem:4.5}{\it{Let  $(  R,\mathfrak{m})$ be a
  local ring of prime characteristic $p>0$, and let $I$ and $J$ be $\mathfrak{m}$-primary ideals of  $R $. Suppose $M $ is a finitely generated $R$-module such that $M_{\mathfrak{p}} $ is free of constant rank $r$ for any $\mathfrak{p} \in \text{Assh} R$.  If $R$ is a domain,  then  $h_{s}(I,J;M)=h_{s}(I,J;R)\cdot r$. }} \end{lem}
\begin{proof}By the associativity formula for the $h$-function (see \cite[Theorem 2.9]{T}), we have

\begin{center}
$\begin{aligned}
h_{s}(I,J;M)&=\mathop{\sum}\limits_{\mathfrak{p} \in \text{Assh}R}h_{s}^{R/\mathfrak{p}}(I(R/\mathfrak{p}),J(R/\mathfrak{p});M)
l_{R_{\mathfrak{p}}}(M_{\mathfrak{p}})\\
&=\mathop{\sum}\limits_{\mathfrak{p} \in \text{Assh}}h_{s}^{R/\mathfrak{p}}(I(R/\mathfrak{p}),J(R/\mathfrak{p});M)
\cdot r\\
&=h_{s}(I,J;R)\cdot r.
\end{aligned}$
\end{center}This completes the proof.
\end{proof}

\begin{prop}\label{prop:4.6}{\it{Let  $(R, \mathfrak{m})$  be a   local ring of prime characteristic $p>0$ and $M $  a finitely generated $R$-module. If $I_{1}$ and $I_{2}$ are $\mathfrak{m}$-primary ideals of  $R $ and $J_{1}=I_{1}\ltimes I_{1}M$, $J_{2}=I_{2}\ltimes I_{2}M$, then

 $\mathrm{(1)}$  $h_{s}(J_{1},J_{2};R \ltimes M)\leq (1+\mu(M)) h_{s}(I_{1},I_{2};R ) $, where $\mu(M)$ denotes the minimal number of generators of $M $.

 $\mathrm{(2)}$  If  $M $  has finite projective dimension $ n $, then  $$h_{s}(J_{1},J_{2};R \ltimes M) =(\mathop{\sum}\limits_{ i=0 }^{n}(-1)^{i}\beta_{i}+1)h_{s}(I_{1},I_{2};R ), $$ where $\beta_{i}$ is $i $-th Betti number  of $M $.

 $\mathrm{(3)}$  If $M_{\mathfrak{p}} $ is free of
constant rank $r$ for every $\mathfrak{p} \in \text{Assh} R$ and  $R$ is a domain, then   $h_{s}(J_{1},J_{2};R \ltimes M) =(r+1)h_{s}(I_{1},I_{2};R ) $.}} \end{prop}
\begin{proof}$\mathrm{(1)}$ Consider the following short exact sequence
$$0 \longrightarrow K  \xrightarrow{ } R^{\mu(M)}  \xrightarrow{ } M  \longrightarrow 0,$$by the additivity of the $h$-function in the short exact sequences in Lemma \ref{lem:2.3}(2), we get
 $$h_{s}(I_{1},I_{2};R^{\mu(M)})=\mu(M)h_{s}(I_{1},I_{2};R)= h_{s}(I_{1},I_{2};K )+h_{s}(I_{1},I_{2};M ) .$$Therefore $$\mu(M)h_{s}(I_{1},I_{2};R)
 \geq h_{s}(I_{1},I_{2};M ) .$$Using Theorem \ref{thm:4.2}, we have $$ h_{s}(J_{1},J_{2};R \ltimes M)=h_{s}(I_{1},I_{2};R ) +h_{s}(I_{1},I_{2};M )\leq(\mu(M)+1)h_{s}(I_{1},I_{2};R).$$

 $\mathrm{(2)}$  Since $M $  has finite projective dimension $ n $, take a minimal free resolution of $M$$$\xymatrix@C=1.5cm{ 0\ar[r]^{} & R^{\beta_{n}}\ar[r]^{}  &\cdots
\ar[r]^{} & R^{\beta_{1}}  \ar[r]^{}& R^{\beta_{0}} \ar[r]^{}& M\ar[r]^{}&0},$$ by the additivity of the $h$-function in the short exact sequences in Lemma \ref{lem:2.3}(2), we get $$h_{s}(I_{1},I_{2};M )=\mathop{\sum}\limits_{ i=0 }^{n}(-1)^{i}\beta_{i}h_{s}(I_{1},I_{2};R ).$$Using Theorem \ref{thm:4.2}, we have
$$h_{s}(J_{1},J_{2};R \ltimes M)
=(\mathop{\sum}\limits_{ i=0 }^{n}(-1)^{i}\beta_{i}+1)h_{s}(I_{1},I_{2};R ).$$

$\mathrm{(3)}$ By  Lemma \ref{lem:4.5}, we have    $h_{s}(I_{1},I_{2};M)= h_{s}(I_{1},I_{2};R)\cdot r$.  Therefore, using Theorem \ref{thm:4.2}, we have $$ h_{s}(J_{1},J_{2};R \ltimes M)=h_{s}(I_{1},I_{2};R ) +h_{s}(I,J;R)\cdot r=(r+1)h_{s}(I_{1},I_{2};R).$$
 This completes the proof.\end{proof}
It is well known that for an ideal $I $  generated by a system of parameters in  $R $, the  (in)equality always holds $$e(I)=e_{HK}(I)\leq \ell_{R}(R/I).$$ To generalize this result, we further apply it to the idealization ring, leading to the following conclusion.

\begin{lem}\label{lem:4.7}{\it{Let  $(R, \mathfrak{m})$  be a   local ring of prime characteristic $p>0$, and let $M $ be a finitely generated $R$-module. If $I_{1}$ and $I_{2}$ are $\mathfrak{m}$-primary ideals of  $R $ and $J_{1}=I_{1}\ltimes I_{1}M$, $J_{2}=I_{2}\ltimes I_{2}M$, then following holds: $$ e_{s}(J_{1},J_{2};R \ltimes M)=e_{s}(I_{1},I_{2};R ) +e_{s}(I_{1},I_{2};M ).$$}} \end{lem}
\begin{proof}
 It follows immediately from Lemma \ref{lem:2.3}(2) that $\text{dim}(R \ltimes M)=\text{dim}(R)$, and since  value $\mathcal{H}_{s}(d)$ depends only on the dimension $d=\text{dim}(R)$ and the parameter $s$,
 the result is a direct consequence of Theorem \ref{thm:4.2}.
\end{proof}

\begin{prop}\label{prop:4.8}{\it{Let  $(R, \mathfrak{m})$  be a   local ring of prime characteristic $p>0$ of dimension $d$, and let $M $ be a finitely generated $R$-module of dimension less than $d$. If $I$ is an  ideal generated by a system of parameters in  $R $, and let $J=I\ltimes IM$,  then following holds:$$ e_{s}(J;R \ltimes M) \leq \ell_{R \ltimes M}(\frac{R \ltimes M}{J} ).$$}} \end{prop}
\begin{proof}
Since  $I$ is an  ideal generated by a system of parameters in  $R $, it follows from \cite[Corollary 3.13]{T} that $ e_{s}(I ) \leq \ell_{R}(R/I)$. Therefore, by Lemmas \ref{lem:4.7} and  \ref{lem:2.3}(3), we obtain
$$ e_{s}(J;R \ltimes M)=e_{s}(I ) \leq \ell_{R}(R/I ).$$Moreover, by \cite[Theorem 3.1]{DM}, there is a isomorphism $$\frac{R \ltimes M}{J}=\frac{R \ltimes M}{I\ltimes IM}\cong\frac{R}{I}\ltimes \frac{M}{IM},$$again, by \cite[Theorem 3.1]{DM} and Lemma \ref{lem:2.2}(1), we have the following short exact sequence
 $$0 \longrightarrow \frac{R}{I}  \xrightarrow{ } \frac{R \ltimes M}{J}  \xrightarrow{ } \frac{M}{IM}  \longrightarrow 0.$$Hence,  we obtain
 $$  \ell_{R}(R/I )\leq \ell_{R}(\frac{R \ltimes M}{J} ) =\ell_{R \ltimes M}(\frac{R \ltimes M}{J} ).$$This completes the proof.
\end{proof}

\bigskip
\section{\bf  Improved Bounds for the  $s$-multiplicity }
As  applications of the above results, in this section, we establish bounds for the $s$-multiplicity in the context of fiber products, idealization rings, and amalgamated duplications.
\subsection{The Taylor-Miller question}
\mbox{}\\

In this subsection, we investigate the Taylor-Miller question in the context of fiber products, idealization rings, and amalgamated duplications, leading to the following result:

\begin{thm}\label{thm:5.1}{\it{Let  $(R, \mathfrak{m})$  be a Cohen-Macaulay   local ring of prime characteristic $p>0$, and let $M $ be a  finitely generated $R$-module.  Then  $$ e_{s}(\mathfrak{m} \ltimes \mathfrak{m}M)\geq 1.$$}} \end{thm}
\begin{proof}By Corollary \ref{cor:4.4} and Lemma \ref{lem:2.3}(4), we obtain $$ e_{s}(\mathfrak{m} \ltimes \mathfrak{m}M)=e_{s}(M )+e_{s}(R ) \geq e_{s}(R )\geq 1.$$This completes the proof.
\end{proof}
The classical result of Taylor and Miller (\cite[Corollary 2.8]{MT}) is recovered by taking
$M $ to be the zero module in the above theorem.

\begin{thm}\label{thm:5.2}{\it{Let $(R, \mathfrak{m}_R, k)$, $(S, \mathfrak{m}_S, k)$ and $(T, \mathfrak{m}_T, k)$ be   local rings of prime characteristic $p>0$. Then the following properties hold for the fiber product  $R \times_TS$:

$\mathrm{(1)}$ If $\text{dim}(R)=\text{dim}(S)=\text{dim}(T)$ and $R$, $T$, $S$ are Cohen-Macaulay, then  $e_{s}(R \times_T S)\geq 1$.

$\mathrm{(2)}$ If $\text{dim}(R)=\text{dim}(S)>\text{dim}(T)$ and $R$, $S$ are Cohen-Macaulay, then   $e_{s}(R \times_T S)\geq 2$.

$\mathrm{(3)}$ If $\text{dim}(R)>\text{dim}(S)>\text{dim}(T)$ and $R $ is Cohen-Macaulay, then   $e_{s}(R \times_T S)\geq 1$.
}} \end{thm}
\begin{proof}$\mathrm{(1)}$ By the assumption and \cite[Proposition 1.7]{AAM}, the fiber product  $R \times_TS$ is Cohen-Macaulay. Hence, it follows from Lemma \ref{lem:2.3}(4) that $e_{s}(R \times_T S)\geq 1$.

$\mathrm{(2)}$ By the assumption and  Lemma \ref{lem:2.3}(4), we obtain $e_{s}(R )\geq 1 $ and $   e_{s}(S)\geq 1$. Corollary \ref{cor:3.5} then immediately implies that $e_{s}(R \times_T S)\geq 2 $.

$\mathrm{(3)}$ From the assumption and  Lemma \ref{lem:2.3}(4), it follows that $e_{s}(R )\geq 1 $.  An immediate application of Corollary  \ref{cor:3.5}  yields  $e_{s}(R \times_T S)\geq 1 $.
\end{proof}

In the above theorem,  statement $\mathrm{(1)}$ represents the classical result of Taylor and Miller  (see Lemma \ref{lem:2.3}(4)). From  statement $\mathrm{(3)}$ and \cite[Proposition 1.7]{AAM},  it follows that $R \times_T S$ is not  Cohen-Macaulay. Therefore, Theorems   \ref{thm:5.1} and \ref{thm:5.2} significantly expand the class of ideals and rings, respectively,  that satisfy the  inequalities predicted by the Taylor-Miller conjecture.

\begin{prop}\label{prop:5.3}{\it{Let  $(R, \mathfrak{m} )$ be a Cohen-Macaulay  local ring of prime characteristic $p>0$, and let $I$ be  a proper ideal of  $R $ such that $\text{dim}(R/I)<\text{dim}(R)$. Then  $$ e_{s}(R \bowtie I)\geq 1.$$}} \end{prop}
\begin{proof}It follows immediately from Corollary  \ref{cor:3.7} and Lemma \ref{lem:2.3}(4).
\end{proof}

\subsection{The $s$-analogue of the Watanabe-Yoshida conjecture}
\mbox{}\\

We extend the result of Taylor and Miller to the $s$-analogue of the Watanabe-Yoshida conjecture for fiber products, amalgamated duplications, and idealization rings using $s$-multiplicity.

\begin{thm}\label{thm:5.4}{\it{Let  $d\geq 2$ and let $k$ be a field of prime characteristic $p>2$. Suppose   $(R, \mathfrak{m}_R, k)$, $(S, \mathfrak{m}_S, k)$ and $(T, \mathfrak{m}_T, k)$  are  local rings of prime characterstic $p>0$, and assume $R$  is a non-regular
 complete intersection  of dimension  $d$. If $\text{dim}(R)>\text{dim}(S)>\text{dim}(T)$, then for any  $s>0 $, we have  $$ e_{s}(R \times_T S) \geq e_{s}(R_{d}).$$}} \end{thm}
\begin{proof}We know that if $\text{dim}(R)>\text{dim}(S)>\text{dim}(T)$, then $e_{s}(R \times_T S)= e_{s}(R )$ by Corollary \ref{cor:3.5}. Furthermore, by \cite[Corollary 3.8]{MT}, we have
 $e_{s}(R )\geq e_{s}(R_{d})$.
\end{proof}

\begin{prop}\label{prop:5.5}{\it{Let  $d\geq 2$ and let $k$ be  a field of prime characteristic $p>2$. Let  $(R, \mathfrak{m},k)$  be  a non-regular  complete intersection  of dimension  $d$ with  prime characterstic $p>0$. If $I$ is a proper ideal, then for any  $s>0 $, we have  $$ e_{s}(R \bowtie I) \geq e_{s}(R_{d}).$$}} \end{prop}
\begin{proof}It follows immediately from Corollary  \ref{cor:3.7} and \cite[Corollary 3.8]{MT}.
\end{proof}

\begin{thm}\label{thm:5.6}{\it{Let  $d\geq 2$ and let $k$ be a field of prime characteristic $p>2$. Let  $(R, \mathfrak{m},k)$  be  a non-regular  complete intersection  of dimension  $d$ with  prime characterstic $p>0$ and  $M $  a finitely generated $R$-module, then for any  $s>0 $, we have  $$ e_{s}(\mathfrak{m} \ltimes \mathfrak{m}R) \geq e_{s}(R_{d}).$$}} \end{thm}
\begin{proof}By Corollary \ref{cor:4.4}, we get $$ e_{s}(\mathfrak{m} \ltimes \mathfrak{m}R)=e_{s}(R ) +e_{s}(M )\geq e_{s}(R ).$$ Furthermore, by \cite[Corollary 3.8]{MT}, we have $e_{s}(R )\geq e_{s}(R_{d})$. By combining these inequalities, we obtain that $ e_{s}(\mathfrak{m} \ltimes \mathfrak{m}R) \geq e_{s}(R_{d})$.
\end{proof}

Theorems   \ref{thm:5.4} and \ref{thm:5.6} significantly expand the class of rings and ideals, respectively, that satisfy the inequalities predicted by the Watanabe-Yoshida conjecture. When $R$ is a complete intersection local ring, the idealization ring $R \ltimes M$ and the fiber product  $R \times_T S$ may fail to be Cohen-Macaulay. Hence, it is unlikely that either ring is regular. Specific examples can be found in \cite[Example 5.9]{VPJ} and  \cite[Example 4.15]{DM}.

\bigskip {\bf Acknowledgement.}
This work was partially supported by the National Natural Science Foundation of China (Grant No. 11261050).

\end{document}